\documentclass{amsart}
\usepackage{amsmath,amsfonts,amsthm,amssymb,indentfirst,epic,url,graphics,needspace}

\theoremstyle{plain}
\newtheorem{theorem}{Theorem}
\newtheorem{lemma}[theorem]{Lemma}
\newtheorem{corollary}[theorem]{Corollary}
\newtheorem{proposition}[theorem]{Proposition}

\newtheorem{definition}[theorem]{Definition}

\newcommand{\Z}{\mathbb{Z}}
\renewcommand{\r}{\mathrm}

\begin{document}

\begin{center}
\texttt{Comments, corrections,
and related references welcomed, as always!}\\[.5em]
{\TeX}ed \today
\vspace{2em}
\end{center}

\title[Completeness results for metrized rings and lattices]%
{Completeness results for metrized rings and lattices}%

\author{George M. Bergman}
\email{gbergman@math.berkeley.edu}
\urladdr{https://math.berkeley.edu/~gbergman}
\address{Department of Mathematics\\University of California\\
Berkeley, CA 94720-3840, USA}

\thanks{
\url{https://arxiv.org/abs/1808.04455}\,.
After publication of this note, updates, errata,
related references etc., if found, will be recorded at
\url{https://math.berkeley.edu/~gbergman/papers/}
}

\dedicatory{Dedicated\\[.1em]
to George Gr\"{a}tzer\\[.1em]
and
to the memory of Jonathan Gleason}



\begin{abstract}
The Boolean ring $B$ of measurable subsets of the unit interval, modulo
sets of measure zero, has proper {\em radical} ideals
(e.g., $\{0\})$ that are closed under the natural metric, but has
no {\em prime} ideals closed under that metric;
hence closed radical ideals are not, in general, intersections
of closed prime ideals.
Moreover, $B$ is known to be complete in its metric.
Together, these facts answer a question posed by J.\,Gleason.
From this example, rings of arbitrary characteristic
with the same properties are obtained.

The result that $B$ is complete in its metric is generalized to
show that if $L$ is a lattice given with a metric satisfying identically
{\em either} the inequality $d(x\vee y,\,x\vee z)\leq d(y,z)$
{\em or} the inequality $d(x\wedge y,\,x\wedge z)\leq d(y,z),$
and if in $L$ every increasing Cauchy sequence
converges and every decreasing Cauchy sequence converges, then every
Cauchy sequence in $L$ converges;
i.e., $L$ is complete as a metric space.

We show by example that if the above inequalities
are replaced by the weaker conditions
$d(x,\,x\vee y)\leq d(x,y),$ respectively
$d(x,\,x\wedge y)\leq d(x,y),$ the completeness conclusion can fail.
We end with two open questions.
\end{abstract}
\maketitle

\section{Overview: a ring-theoretic question, culminating in a lattice-theoretic result}\label{S.intro}

A standard result of ring theory says that if $I$ is an ideal
of a commutative ring $R,$ then the nil radical of $I$ (the
ideal of elements having some power in $I)$
is the intersection of the prime ideals of $R$ containing~$I$
\cite[Proposition~10.2.9, p.\,352]{PMC}.

Jonathan Gleason\footnote{
	Jonathan Gleason, then a graduate student in mathematics at the
	University of California, Berkeley, raised this question
	shortly before his unexpected tragic death in January, 2018.}
(personal communication) asked the present author
about a possible generalization of that result.
Namely, suppose $R$ is a topological commutative ring.
For any ideal $I$ of $R,$
let $\sqrt I$ denote the least {\em closed} ideal
$J\supseteq I$ such that $J$ contains
every element $x$ such that $x^n\in J$ for some $n\geq 1.$
Must $\sqrt I$ be the intersection of
all closed {\em prime} ideals containing $I$?
If not in general, does this become true
if $R$ is complete with respect to the given topology?

We shall see that the answer
is negative:  If $R$ is the Boolean ring of measurable
subsets of the unit interval modulo sets of measure zero,
topologized using the metric given by
the measure of the symmetric difference of such sets,
then $R$ is complete in that metric, and $\{0\} = \sqrt{\{0\}}$
(defined as above); but $R$ has no closed prime ideals, so
$\{0\}$ is not an intersection of such ideals.
We give the details in~\S\ref{S.Bool}, and note in~\S\ref{S.nonBool}
how to get, from this characteristic-$\!2$ example,
examples of arbitrary characteristic.

The one not-so-obvious property of our example
is the completeness of $B$ as a metric space.
In~\S\ref{S.lattice} (which
is independent of \S\S\ref{S.Bool}-\ref{S.nonBool})
we note that this can be deduced from a standard result
of measure theory, and then prove a general result on when
a metrized {\em lattice} is complete, which yields an alternate proof.

In~\S\ref{S.ceg} we give a curious counterexample to that
general completeness result under a weakened hypothesis.


\section{The Boolean example}\label{S.Bool}

Most of the desired properties of the example sketched above
are straightforward to verify.

Recall that for any set $X,$ the subsets of $X$ form
a Boolean ring under the operations
\begin{equation}\begin{minipage}[c]{35pc}\label{d.Boole}
$0=\emptyset,\qquad 1=X,\qquad
S+T=S\cup T \setminus (S\cap T),\qquad S\,T=S\cap T.$
\end{minipage}\end{equation}

Now let $B_0$ be the set of measurable subsets of the unit interval
$[0,1],$ and for $S\in B_0,$ let $\mu(S)\in[0,1]$ be its measure.
$B_0$ clearly forms a subring of the Boolean ring of subsets
of $[0,1],$ and
for any $S,\,T\in B_0,$ we see from the above definition of $S+T$ that
\begin{equation}\begin{minipage}[c]{35pc}\label{d.muS+muT}
$\mu(S\cup T)
\ =\ \mu(S+T)+\mu(S\cap T)
\ =\ \mu(S)+\mu(T) - \mu(S\cap T).$
\end{minipage}\end{equation}

For $S,\,T\in B_0,$ let
\begin{equation}\begin{minipage}[c]{35pc}\label{d.d0}
$d_0(S,\,T)\ =\ \mu(S+T).$
\end{minipage}\end{equation}
Then $d_0$ is a pseudometric, i.e., for all $S,\,T,\,U\in B_0,$
\begin{equation}\begin{minipage}[c]{35pc}\label{d.d0geq0}
$d_0(S,\,T)\ \geq\ 0,$
\end{minipage}\end{equation}
\begin{equation}\begin{minipage}[c]{35pc}\label{d.d0SS}
$S=T\ \implies\ d_0(S,\,T)=0,$
\end{minipage}\end{equation}
\begin{equation}\begin{minipage}[c]{35pc}\label{d.d0TS}
$d_0(T,\,S)\ =\ d_0(S,\,T),$
\end{minipage}\end{equation}
\begin{equation}\begin{minipage}[c]{35pc}\label{d.d0STU}
$d_0(S,\,T)+d_0(T,\,U)\ \geq\ d_0(S,U).$
\end{minipage}\end{equation}
Here~\eqref{d.d0geq0}-\eqref{d.d0TS} are immediate.
To get~\eqref{d.d0STU}, note that the second
equality of~\eqref{d.muS+muT} gives
the inequality $\mu(S)+\mu(T) \geq \mu(S+T).$
Putting $S+T$ and $T+U$ in place of $S$ and $T$
in that relation gives~\eqref{d.d0STU}.

The definition~\eqref{d.d0}, and the identities of Boolean rings,
show that the Boolean operations behave nicely under~$d_0$:
\begin{equation}\begin{minipage}[c]{35pc}\label{d.d0_ct}
$d_0(S{+}U,\,T{+}U)\ =\ d_0(S,\,T),$
\end{minipage}\end{equation}
\begin{equation}\begin{minipage}[c]{35pc}\label{d.d0SU,TU}
$d_0(SU,\,TU)\ \leq\ d_0(S,\,T).$
\end{minipage}\end{equation}
These relations show in particular that if $S$ and $T$ are close
to one another under $d_0,$ then the results of adding
$U$ to $S$ and $T$ are also
close to one another, as are the results of multiplying
$S$ and $T$ by $U,$ in each case in a uniform way, whence
addition and multiplication are continuous with respect to $d_0.$

Now let $B$ be the quotient ring
\begin{equation}\begin{minipage}[c]{35pc}\label{d.B}
$B\ =\ B_0\,/\,\{S\mid\mu(S)=0\}.$
\end{minipage}\end{equation}
Let us write $[S]\in B$ for the residue class of
an element $S\in B_0,$ and define
\begin{equation}\begin{minipage}[c]{35pc}\label{d.d}
$d([S],[T])\ =\ d_0(S,\,T)\ =\ \mu(S+T)$ \ for \ $[S],\,[T]\in B.$
\end{minipage}\end{equation}
Then~\eqref{d.d0geq0}-\eqref{d.d0SU,TU} clearly carry over to $d,$ with
the ``$\!\!\implies\!\!$'' of~\eqref{d.d0SS} strengthened
to ``$\!\!\iff\!\!$''.
Thus, $d$ is a metric, and

\begin{lemma}\label{L.B}
The operations of the Boolean ring $B$ of measurable subsets
of $[0,1]$ modulo sets of measure zero are continuous in
the metric $d$ of~\eqref{d.d}.\qed
\end{lemma}

Being a Boolean ring, $B$ has no nonzero nilpotent elements,
hence the ideal $\{0\}$ of $B$ trivially contains
every $x\in B$ such that $x^n\in\{0\}$ for some $n\geq 1;$
and being a singleton, $\{0\}$ is closed in the metric topology.
So $\sqrt{\{0\}}=\{0\}$ under the definition
of $\sqrt{I}$ suggested by Gleason.

Is $\{0\}$ an intersection of closed prime ideals?
A negative answer follows from

\begin{lemma}\label{L.no_closed_primes}
Every prime ideal $P$ of the ring $B$ has for topological closure
the whole ring $B.$

Hence $B$ has no closed prime ideals.
\end{lemma}

\begin{proof}
Given a prime ideal $P,$ let us first show that $P$ has elements
arbitrarily close to $1,$ i.e., elements $[1+U]$ such that $U$
has arbitrarily small measure.
To do this, we shall show that whenever $[1+U]\in P,$
there exists $[1+S]\in P$ with $\mu(S)=\mu(U)/2.$

Indeed, let us write $U$ as the union of two disjoint measurable
subsets $S$ and $T,$ each of measure $\mu(U)/2.$
(E.g., one can take $S=U\cap [0,t]$ and
$T=U\cap (t,1]$
for appropriate $t\in[0,1].$
Such a $t$ exists by continuity of $\mu(U\cap[0,t])$ in $t.)$
Since $[1+S][1+T]=[1+U]\in P,$ one of $[1+S],\,[1+T]$
belongs to $P.$

So $P$ indeed has elements arbitrarily close to $1.$
Multiplying
arbitrary $[V]\in B$ by such elements, we see that $P$ has
elements arbitrarily close to $[V];$
so the closure of $P$ contains every $[V]\in B,$ as claimed.

Hence no prime ideal $P$ is itself closed, giving the final assertion.
\end{proof}

So $\{0\} = \sqrt{\{0\}}$ is not an intersection of
closed prime ideals.
Since $B$ is complete in the metric $\mu$
(to be proved in two ways in \S\ref{S.lattice}), $B$
answers the strongest form of Gleason's question.

(To be precise, Gleason's question
concerned completeness of a topological
ring $R$ in the {\em uniform structure} arising from
additive translates of neighborhoods of $0.$
Our metrized ring has additive-translation-invariant
metric by~\eqref{d.d0_ct}, so completeness in
the uniform structure so arising from the metric topology
is equivalent to completeness in the metric.)

\section{Non-Boolean algebras}\label{S.nonBool}

Is the behavior of above example limited to Boolean rings,
or perhaps to rings of finite characteristic?
No.
We note below how to generalize the construction of the preceding
section to algebras in the sense of General Algebra
(a.k.a.\ Universal Algebra),
and observe that when the algebras in question are rings
(of arbitrary characteristic), these give more varied examples
of the properties proved for $B.$

We start with the analog of $B_0.$

\begin{definition}\label{D.X'}
For $X$ a set, let $X^{[0,1]}$ denote the set
of all $\!X\!$-valued functions on the unit interval, and
for each $f\in X^{[0,1]}$ and $x\in X,$ let
\begin{equation}\begin{minipage}[c]{35pc}\label{d.f_x}
$f_x\ =\ \ \{t\in[0,1]\mid f(t)=x\}.$
\end{minipage}\end{equation}

Let $X'$ denote the subset of $X^{[0,1]}$ consisting
of those $f$ such that
\begin{equation}\begin{minipage}[c]{35pc}\label{d.f_x_meas}
for all $x\in X,$ the set $f_x\subseteq[0,1]$ is measurable,
\end{minipage}\end{equation}
and
\begin{equation}\begin{minipage}[c]{35pc}\label{d.X'_ctb}
the image of $f,$ that is,
$\{x\in X\mid f_x\neq\emptyset\},$ is countable \textup{(}i.e.,
finite or countably infinite\textup{)}.
\end{minipage}\end{equation}

For $f,g\in X',$ let
\begin{equation}\begin{minipage}[c]{35pc}\label{d.d'}
$d'(f,\,g)\ =\ \mu(\{t\in[0,1]\mid f(t)\neq g(t)\})\ =
\ (\sum_{x\in X}\,d_0(f_x,\,g_x))\,/\,2,$
\end{minipage}\end{equation}
where $d_0$ is the pseudometric on measurable subsets of $[0,1]$
defined in~\eqref{d.d0}.
\end{definition}

The final equality of~\eqref{d.d'} is, intuitively, a
consequence of the fact that every $t\in [0,1]$ such that
$f(t)\neq g(t)$ contributes twice to the summation in the final term:
via the summand with $x=f(t)$ and the summand with $x=g(t).$
That idea is easily formalized to show that
that sum is indeed twice the middle term of~\eqref{d.d'}.

\begin{lemma}\label{L.X'}
For any set $X,$ the function $d'$ defined
by~\eqref{d.d'} is a pseudometric on $X'.$

For any finitary operation $u: X^n\to X,$ the induced pointwise
operation $u^{[0,1]}: (X^{[0,1]})^n\to X^{[0,1]}$ carries $(X')^n$ to
$X',$ and for $f^{(0)},\dots,f^{(n-1)},\,g^{(0)},\dots,g^{(n-1)}\in X'$
we have
\begin{equation}\begin{minipage}[c]{35pc}\label{d.d'_and_u}
$d'(u^{[0,1]}(f^{(0)},\dots,f^{(n-1)}),\,
u^{[0,1]}(g^{(0)},\dots,g^{(n-1)}))\ \leq
\ d'(f^{(0)},\,g^{(0)}) + \dots + d'(f^{(n-1)},\,g^{(n-1)}).$
\end{minipage}\end{equation}
\end{lemma}

\begin{proof}
That $d'$ is a pseudometric is straightforward.
(The triangle inequality is verified using the
rightmost expression of~\eqref{d.d'}
and the fact that $d_0$ is a pseudometric.)

We need to know next that given $u: X^n\to X,$
and $f^{(0)},\dots,f^{(n-1)}\in X',$ we have
$u^{[0,1]}(f^{(0)},\dots,\linebreak[0]f^{(n-1)})\in X'.$
Note that for each $x\in X,$ $u^{[0,1]}(f^{(0)},\dots,f^{(n-1)})_x$
will be the union, over all $\!n\!$-tuples
$(x_0,\dots,x_{n-1})$ satisfying $u(x_0,\dots,x_{n-1})=x,$ of the sets
\begin{equation}\begin{minipage}[c]{35pc}\label{d.n-fold_cap}
$f^{(0)}_{x_0}\cap\dots\cap f^{(n-1)}_{x_{n-1}}.$
\end{minipage}\end{equation}
Now for each $i\in n,$ only countably many values
of $x_i$ make $f^{(i)}_{x_i}$ nonempty, so only countably
many $\!n\!$-tuples $(x_0,\dots,x_{n-1})$
make the intersection~\eqref{d.n-fold_cap} nonempty;
and by~\eqref{d.f_x_meas}, each of those
$\!n\!$-fold intersections is measurable;
so for each $x,$ $u^{[0,1]}(f^{(0)},\dots,f^{(n-1)})_x$ is a countable
union of measurable sets, hence measurable; i.e.,
$u^{[0,1]}(f^{(0)},\dots,f^{(n-1)})$ satisfies
the condition of~\eqref{d.f_x_meas}.
It also satisfies the condition of~\eqref{d.X'_ctb},
since the countably many cases where~\eqref{d.n-fold_cap}
is nonempty lead to only countably many possibilities for the element
$u(x_0,\dots,x_{(n-1)}).$
So $u^{[0,1]}$ carries $(X')^n$ to $X'.$

The inequality~\eqref{d.d'_and_u} follows from the fact that
$u(f^{(0)},\dots,f^{(n-1)})$ and
$u(g^{(0)},\dots,g^{(n-1)})$
can differ only at points $t\in[0,1]$ where $f^{(i)}$ and $g^{(i)}$
differ for at least one~$i.$
\end{proof}

We now want to deduce the corresponding
results with the set of functions $X'$
replaced by the set of equivalence classes of such functions under the
relation of differing on a set of measure zero.
We will need the following observation.

\begin{lemma}\label{L.parttn}
As in {\rm\S\ref{S.Bool}}, let $B_0$ denote the Boolean ring
of measurable subsets of $[0,1].$

Let $S_0,$ $S_1,\,\dots$ be a countable {\rm(}i.e., finite or countably
infinite{\rm)} family of elements of $B_0$ such that
\begin{equation}\begin{minipage}[c]{35pc}\label{d.muSiSj}
$\mu(S_i\cap S_j)\ =\ 0$ \ whenever $i\neq j,$
\end{minipage}\end{equation}
and
\begin{equation}\begin{minipage}[c]{35pc}\label{d.sum_muSi}
$\sum_i\,\mu(S_i)\ =\ 1.$
\end{minipage}\end{equation}

Then there exist $T_0,$ $T_1,\ldots\in B_0$ such that
\begin{equation}\begin{minipage}[c]{35pc}\label{d.muSi+Ti}
$d_0(S_i,\,T_i)\ =\ 0\quad (i=0,1,\dots),$
\end{minipage}\end{equation}
and the $T_i$ partition $[0,1],$ i.e., satisfy the two conditions
\begin{equation}\begin{minipage}[c]{35pc}\label{d.TiTj}
$T_i\cap T_j\ =\ \emptyset$ whenever $i\neq j$
{\rm(}cf.~\eqref{d.muSiSj}\rm{)}
\end{minipage}\end{equation}
and
\begin{equation}\begin{minipage}[c]{35pc}\label{d.cupTi}
$\bigcup_i\,T_i\ =\ [0,1]$
{\rm(}cf.~\eqref{d.sum_muSi}\rm{)}.
\end{minipage}\end{equation}
\end{lemma}

\begin{proof}
Let
\begin{equation}\begin{minipage}[c]{35pc}\label{d.Ti}
$T_i\ =\ S_i\,\setminus\,\bigcup_{0\leq j<i}\,S_j$ \ for \ $i>0,$
\end{minipage}\end{equation}
and
\begin{equation}\begin{minipage}[c]{35pc}\label{d.T0}
$T_0\ =\ [0,1]\,\setminus\,\bigcup_{i>0}\,T_i.$
\end{minipage}\end{equation}

These sets are clearly measurable and partition~$[0,1].$

Since the sets $S_j$ whose members are removed from $S_i$
in~\eqref{d.Ti} have, by~\eqref{d.muSiSj}, only a set
of measure zero in common with $S_i,$ we see that for
$i>0,$ $T_i$ differs from $S_i$ in a set of measure zero,
giving~\eqref{d.muSi+Ti} for such~$i.$
Also by~\eqref{d.Ti}, no $T_i$ with $i>0$ contains elements of $S_0,$
so by~\eqref{d.T0}, $T_0\supseteq S_0;$
hence to prove the $i=0$ case of~\eqref{d.muSi+Ti}, it
suffices to show that $\mu(T_0)=\mu(S_0).$
To do this, note that since the $T_i$ partition $[0,1],$ we have
$\sum_{i\geq 0}\,\mu(T_i) = \mu([0,1]) = 1 =
\sum_{i\geq 0}\,\mu(S_i)$ by~\eqref{d.sum_muSi}.
If we subtract from that
relation the equations $\mu(T_i)=\mu(S_i)$ for all $i>0,$
which follow from the cases of~\eqref{d.muSi+Ti} already obtained,
we get the desired $i=0$ case.
\end{proof}

Now -- still assuming the completeness of $B,$ to be obtained
in the next section -- we can get

\begin{proposition}\label{P.X^*}
For $X$ a set, let $X^*$ denote the quotient of $X'$
\textup{(}defined in Definition~\ref{D.X'}\textup{)} by
the equivalence relation $d'(f,\,g)=0,$ and let $d^*$ be the metric
on $X^*$ induced by~$d'.$

Then $X^*$ under the metric $d^*$ is a complete metric space.
\end{proposition}

\begin{proof}
Consider any Cauchy sequence
$[f^{(0)}],\,[f^{(1)}],\,\dots\in X^*,$ where
$f^{(0)},\,f^{(1)},\,\dots\in X'.$
For each $n,$ the set of elements $x\in X$ such that $f^{(n)}_x$ is
nonempty is countable by~\eqref{d.X'_ctb}; hence there exists a
countable (possibly finite) list of distinct such elements:
\begin{equation}\begin{minipage}[c]{35pc}\label{d.x0x1...}
$\{x_0,\,x_1,\dots\}\ =
\ \{x\in X\mid(\exists\,n)\,f^{(n)}_x\neq\emptyset\}.$
\end{minipage}\end{equation}

Now (writing $d_0$ and $d,$ as in the preceding section,
for our pseudometric on $B_0$ and metric on $B),$ we have
for all $i,\,m,\,n,$
\begin{equation}\begin{minipage}[c]{35pc}\label{d.df_vs_df_x}
$d([f^{(m)}_{x_i}],\,[f^{(n)}_{x_i}])\ =
\ d_0(f^{(m)}_{x_i},\,f^{(n)}_{x_i})\ \leq\ d'(f^{(m)},\,f^{(n)})
\ =\ d^*([f^{(m)}],\,[f^{(n)}]);$
\end{minipage}\end{equation}
hence the Cauchyness of the sequence of elements $[f^{(n)}]\in X^*$
implies, for each $x_i,$ the Cauchyness of the sequence of
$[f^{(n)}_{x_i}]\in B.$
Hence by the completeness of $B,$ for each $i$ the sequence
$[f^{(0)}_{x_i}],\,[f^{(1)}_{x_i}],\,\dots$
converges to an element which we shall write $[S_i],$
choosing an arbitrary representative $S_i\in B_0$
of the limit of that sequence in $B.$

I claim that these sets $S_i$
satisfy~\eqref{d.muSiSj} and~\eqref{d.sum_muSi}.
To get the first of these equations for given $i\neq j,$ note that
for any $\varepsilon>0,$ one can choose $n$ such that
$d_0(f^{(n)}_{x_i},\,S_i)<\varepsilon/2$
and $d_0(f^{(n)}_{x_j},\,S_j)<\varepsilon/2.$
Since $f^{(n)}_{x_i}$ and $f^{(n)}_{x_j}$
are disjoint, we see that $S_i$ and $S_j$
intersect in a set of measure at most $\varepsilon.$
Since this holds for all $\varepsilon>0,$ they
intersect in a set of measure $0.$

To get~\eqref{d.sum_muSi}, note that for any $\varepsilon>0$
we may choose $m$ such that for all
$n\geq m$ we have $d^*([f^{(m)}],\,[f^{(n)}])<\varepsilon/3,$
equivalently,
\begin{equation}\begin{minipage}[c]{35pc}\label{d.m}
$d'(f^{(m)},\,f^{(n)})\ <\ \varepsilon/3.$
\end{minipage}\end{equation}
Since $f_{x_0},\,f_{x_1},\,\dots$ partition $[0,1],$
we can also choose $j$ such that
\begin{equation}\begin{minipage}[c]{35pc}\label{d.j}
$\mu(f^{(m)}_{x_0})+\dots+\mu(f^{(m)}_{x_j})\ \geq\ 1-\varepsilon/3.$
\end{minipage}\end{equation}
For $n\geq m,$~\eqref{d.m} guarantees that
$d_0(f^{(m)}_{x_0},\,f^{(n)}_{x_0})+\dots+
d_0(f^{(m)}_{x_j},\,f^{(n)}_{x_j})<2\varepsilon/3$ (see~\eqref{d.d'}),
equivalently,
$d([f^{(m)}_{x_0}],\,[f^{(n)}_{x_0}])+\dots+
d([f^{(m)}_{x_j}],\,[f^{(n)}_{x_j}])<2\varepsilon/3,$
hence passing to the limit as $n\to\infty,$
$d([f^{(m)}_{x_0}],\,[S_0])+\dots+
d([f^{(m)}_{x_j}],\,[S_j])\leq 2\varepsilon/3;$
and combining with~\eqref{d.j} we get
$\mu(S_0)+\dots+\mu(S_j)\geq 1-\varepsilon,$
equivalently,
$\mu(S_0\cup\dots\cup S_j)\geq 1-\varepsilon,$
Since this holds for all $\varepsilon,$ we
have $\sum_i\mu(S_i)= \mu(\bigcup_i S_i)\geq 1;$
and since a subset of $[0,1]$ cannot have measure
larger than~$1,$ we get~\eqref{d.sum_muSi}.

Lemma~\ref{L.parttn} now gives us a partition of $[0,1]$
into sets $T_i$ which differ from the $S_i$ by sets of measure zero.
If we define $f\in X'$ by
\begin{equation}\begin{minipage}[c]{35pc}\label{d.g}
$f_{x_i}\ =\ T_i$ \ for all $i$ (whence by~\eqref{d.cupTi},
$f_x=\emptyset$ for all $x$ not of the form $x_i),$
\end{minipage}\end{equation}
then $[f]\in X^*$ is a limit of the given Cauchy
sequence $[f^{(0)}],\,[f^{(1)}],\,\dots\,,$ proving completeness.
\end{proof}

Remark:  If in~\eqref{d.X'_ctb} we had allowed uncountable
cardinalities, we would not have been able to use basic properties
of measure, e.g., in concluding that the set in the middle
term of~\eqref{d.d'} was measurable, and in
proving in Lemma~\ref{L.X'} that $u^{[0,1]}$ carries $(X')^n$ to $X'.$
On the other hand, if we had required the set of
$x$ making $f_x$ nonempty to be finite, our $X^*$ would not have been
complete, except in the case where $X$ was finite.
So countability is the only choice that gives our construction
$X^*$ the desired properties.

We have not yet called on~\eqref{d.d'_and_u}.
It implies that our construction
behaves nicely on algebras:

\begin{proposition}\label{P.alg}
Suppose $A$ is an algebra in the sense of General Algebra,
that is, a set given with a \textup{(}finite or infinite\textup{)}
family of operations, each of finite arity.

Then for each operation $u: A^n\to A$ of $A,$ the operation
$u^*$ of $A^*$ described by
\begin{equation}\begin{minipage}[c]{35pc}\label{d.u^*}
$u^*([f_0],\dots,[f_{n-1}])\ =\ [u^{[0,1]}(f_0,\dots,f_{n-1})]$
\end{minipage}\end{equation}
is well-defined,
and uniformly continuous in the metric $d^*;$ indeed, it satisfies
Lipschitz condition
\begin{equation}\begin{minipage}[c]{35pc}\label{d.d^*_and_u}
$d^*(u^*([f^{(0)}],\dots,[f^{(n-1)}]),\,
u^*([g^{(0)}],\dots,[g^{(n-1)}]))\\[.3em]
\hspace*{3em}\leq\ d^*([f^{(0)}],\,[g^{(0)}]) + \dots +
d^*([f^{(n-1)}],\,[g^{(n-1)}]).$
\end{minipage}\end{equation}

The algebra $A^*$ satisfies all identities satisfied by $A.$
In fact, every finite set of elements of $A^*$ is contained
in a subalgebra of $A^*$ isomorphic to a countable direct product
of copies of $A.$
\end{proposition}

\begin{proof}[Sketch of proof]
By the case of~\eqref{d.d'_and_u}
where the $d'(f^{(i)},g^{(i)})$ are all zero,
the operations $u^{[0,1]}$ of $A'$
respect the equivalence relation used in defining $A^*,$
so~\eqref{d.u^*} gives well-defined operations.
The general case of~\eqref{d.d'_and_u} then
gives the Lipschitz inequality~\eqref{d.d^*_and_u}, and in particular,
continuity.
Finally, given any finite family of elements
$[f^{(0)}],\dots,[f^{(N-1)}]$ of $A^*,$ the countably many
nonempty sets $f^{(i)}_a$ $(0\leq i<\nolinebreak N,$ $a\in A)$ yield
a decomposition of $[0,1]$ into countably many
intersections as in~\eqref{d.n-fold_cap}, on each of which all
of $f^{(0)},\dots,f^{(N-1)}$ are constant.
Dropping those intersections that have measure zero,
and looking at the algebra
of members of $A'$ that are constant on the remaining
countably many subsets, we see that this
contains $f^{(0)},\dots,f^{(N-1)},$ and has as its image
in $A^*$ a subalgebra isomorphic to a countable direct product
of copies of $A.$
\end{proof}

Finally, some observations specific to rings:

\begin{proposition}\label{P.ring}
Let $A$ be an associative unital ring.
Then in the complete metrized ring $A^*$ arising by the construction of
Proposition~\ref{P.alg}, the closure of every prime ideal is all
of $A^*;$ hence $A^*$ has no closed prime ideals.

On the other hand, if $A$ has no nonzero nilpotent elements, then
the topological radical $\sqrt{\{0\}},$ defined as
in~{\rm\S\ref{S.intro}}, is $\{0\}.$
\end{proposition}

\begin{proof}[Sketch of proof.]
Let $P$ be a prime ideal of $A^*.$
Writing $1_S$ for the characteristic function
with values in $\{0,1\}\subseteq A$ of a
subset $S\subseteq[0,1],$ let us show that
$P$ contains elements $[1_{[0,1]\setminus U}]$ for
sets $U$ of arbitrarily small positive measure.
Clearly, it contains $[1_{[0,1]\setminus U}]$ for $U=[0,1].$
Given any $U$ such that $[1_{[0,1]\setminus U}]\in P,$
let us partition $U$
into two measurable subsets $S$ and $T$ of equal measure.
Since $\!\{0,1\}\!$-valued functions are central in $A',$
so are their images in $A^*,$ so we have
\begin{equation}\begin{minipage}[c]{35pc}\label{d.charfns}
$[1_{[0,1]\setminus S}]\,A^*\,[1_{[0,1]\setminus T}]
\,=\,[1_{[0,1]\setminus S}]\,[1_{[0,1]\setminus T}]\,A^*
\,=\,[1_{[0,1]\setminus(S\cup T)}]\,A^*
\,=\,[1_{[0,1]\setminus U}]\,A^*\,\subseteq P,$
\end{minipage}\end{equation}
so as $P$ is prime, one of
$[1_{[0,1]\setminus S}],$ $[1_{[0,1]\setminus T}]$
belongs to $P;$ so we have cut in half the measure of
our set $U$ with $[1_{[0,1]\setminus U}]\in P.$
Since $d([1_{[0,1]\setminus U}],\,1)=\mu(U),$ we have,
as in the proof of Lemma~\ref{L.no_closed_primes},
found elements of $P$ arbitrarily close to $1,$ and
can deduce that the closure of $P$ is all of~$A^*.$

The final assertion is straightforward.
\end{proof}

Thus, for $A$ a commutative ring without nilpotents,
the rings $A^*$ generalize
the properties of the example $B$ of the preceding section.
That $B$ is, of course, the case of this construction with $A=\Z/2\Z.$

Remark: The development of the above results
in terms of measurable $\!X\!$-valued
functions on $[0,1],$ modulo
disagreement on sets of measure zero, feels artificial.
Surely one should be able to perform our constructions
abstractly in terms of the set $X,$ the Boolean ring $B,$
and the real-valued function on $B$ induced by the measure on $[0,1],$
and then generalize it to get such results with $B$ replaced
by any Boolean ring with an appropriate real-valued function.

If we were interested in maps $[0,1]\to X$ assuming
only {\em finitely many} values, then
the analog of $X^*$ could be described as the set of continuous
functions from the Stone space of $B$ to the discrete space~$X.$
But for maps allowed to assume countably many values,
the function corresponding to the metric
seems to be needed in defining $X^*.$
I leave the proper formulation and generalization
of that construction to experts in the subject.
Cf.~\cite[Chapter~31]{Fremlin}.

In contrast, the results of the next section will be
obtained in a satisfyingly general context.

\section{Completeness}\label{S.lattice}

When I first suggested the Boolean ring $B$ of measurable subsets of
$[0,1],$ modulo sets of measure zero, as an answer to J.\,Gleason's
question, the one property not clear to me was completeness
in the natural metric, though it seemed intuitively likely.

One might naively hope to prove completeness by
showing that every sequence of
measurable subsets of $[0,1]$ whose images in $B$
form a Cauchy sequence
``converges almost everywhere'' on $[0,1];$ i.e.,
that almost every $t\in[0,1]$ belongs either to all
but finitely many members of the sequence, or to only finitely many.
But this is not so; a counterexample
\cite[Exercise 22(6), p.94]{Halmos}
is the sequence whose first term
is $[0,1],$ whose next two are $[0,\,1/2]$ and $[1/2,\,1],$
whose next three are $[0,\,1/3],$ $[1/3,\,2/3],$ $[2/3,\,1],$
and, generally, whose $\!1{+}2{+}\ldots{+}(n{-}1){+}i\!$-th term for
$1\leq i\leq n$ is $[(i-1)/n,\,i/n].$
The measures of these sets approach zero, so the
sequence approaches $\emptyset$ in our metric; but clearly every
$t\in[0,1]$ occurs in infinitely many of these sets.
Looking at this example, one might still hope
that given a Cauchy sequence in $B,$ almost every
$t\in [0,1]$ has the property that the terms $S_i$
which contain $t$ are either ``eventually scarce'',
or have eventually scare complement.
But this, too, fails;
to see this, take the above example,
and ``stretch it out'' by repeating the $\!m\!$-th term $2^m$
times successively, for each~$m.$

However, an online search turned up a proof of the desired
completeness statement in a set of
exercises~\cite{Mennucci} (in particular point~6 on p.\,2).
I cited that in the first draft of this note as the only reference
for the result that I could find.
David Handelman then pointed out that the desired statement follows
immediately from the standard fact
that $L^1$ of the unit interval is complete in its natural
metric (\cite[Theorem VI.3.4, p.133]{Lang},
\cite[Theorem 22.E, p.93]{Halmos}), on identifying measurable sets
with their characteristic functions.
(And indeed, in \cite[Exercise~40(1), p.169]{Halmos},
the reader is asked to deduce the
result we want from that result about~$L^1.)$
Subsequently, Hannes Thiel
pointed me to a result of the desired sort
proved for a large class of Boolean rings
with measure-like $\![0,1]\!$-valued
functions \cite[Theorem 323G(c)]{Fremlin}.
(The condition there called {\em localizability}
means, roughly, that the Boolean ring has ``enough''
elements of finite measure, and has joins of arbitrary subsets.)

In all these sources, the key to the proof of
completeness is to pass from an arbitrary Cauchy sequence to
a subsequence with the property that the distance between the
$\!i\!$-th and $\!i{+}1\!$-st terms is $\leq 2^{-i}.$
Rather magically, a sequence with this property does indeed converge
almost everywhere, giving a limit of the original Cauchy sequence.

In fact, this trick can be abstracted from the context
of measure theory to that of lattices (or even semilattices)
as in the next theorem, from which we will recover, as a corollary,
the result on measurable sets modulo null sets.

Since we no longer need the notation ``$f_x$'' of the
preceding section for the point-set at which a function
takes on the value $x,$ we will henceforth use subscripts in the
conventional way to index terms of sequences.

We remark that the condition that a metrized lattice
be complete as a metric space, obtained in the theorem,
is independent of its completeness
as a lattice, i.e., the existence of least upper bounds and greatest
lower bounds of not necessarily finite subsets (though the
condition that {\em certain} infinite least upper bounds and greatest
lower bounds exist will be key to the argument).
For instance, any lattice, given with the metric that
makes $d(x,y)=1$ whenever $x\neq y,$ is complete as a metric
space, and, indeed, satisfies the hypotheses
of the next theorem, but need not be complete as a lattice.
Inversely, the totally ordered
subset of the real numbers $\{-2\}\cup(-1,1)\cup\{2\}$ is complete
as a lattice, but not as a metric space.

\begin{theorem}\label{T.lat}
Let $L$ be a lattice, whose underlying set is given with
a metric $d$ which satisfies identically at least one of the
inequalities
\begin{equation}\begin{minipage}[c]{35pc}\label{d.dv}
$d(x\vee y,\,x\vee z)\ \leq\ d(y,\,z)$ $(x,y,z\in L),$
\end{minipage}\end{equation}
\begin{equation}\begin{minipage}[c]{35pc}\label{d.dw}
$d(x\wedge y,\,x\wedge z)\ \leq\ d(y,\,z)$ $(x,y,z\in L)$
\end{minipage}\end{equation}
\textup{(}or, more generally, let $L$ be an upper semilattice
satisfying~\eqref{d.dv}, or a lower semilattice
satisfying~\eqref{d.dw}\textup{)}.

Suppose moreover that in $L$
\begin{equation}\begin{minipage}[c]{35pc}\label{d.incrCauchy}
every increasing Cauchy sequence
$x_0\leq x_1\leq\ldots\leq x_n\leq\dots$ converges,
\end{minipage}\end{equation}
and likewise
\begin{equation}\begin{minipage}[c]{35pc}\label{d.decrCauchy}
every decreasing Cauchy sequence
$x_0\geq x_1\geq\ldots\geq x_n\geq\dots$ converges.
\end{minipage}\end{equation}

Then every Cauchy sequence in $L$ converges; i.e.,
$L$ is complete as a metric space.
\end{theorem}

\begin{proof}
It suffices to prove the case where $L$ is an upper
semilattice satisfying~\eqref{d.dv}, since this includes
the case where $L$ is a lattice satisfying~\eqref{d.dv},
while the cases where $L$ is a lower semilattice or lattice
satisfying~\eqref{d.dw} follow by duality.
So assume $L$ such an upper semilattice.

In proving $L$ complete, it suffices to show
convergence of sequences $x_0, x_1,\dots$ such that
\begin{equation}\begin{minipage}[c]{35pc}\label{d.sum_d()}
$\sum_{i\geq 0}\,d(x_i,\,x_{i+1})\ <\ \infty,$
\end{minipage}\end{equation}
since every Cauchy sequence has such a subsequence (e.g., one chosen
so that $d(x_i,\,x_{i+1})\leq 2^{-i},$
as in \cite{Halmos}, \cite{Lang} and \cite{Mennucci}),
and if a subsequence of a Cauchy sequence converges, so does
the whole sequence.

So let the sequence $x_0, x_1,\dots$ satisfy~\eqref{d.sum_d()},
and let us define
\begin{equation}\begin{minipage}[c]{35pc}\label{d.xij}
$x_{h,j}\ =\ x_h\vee x_{h+1}\vee\dots\vee x_j$ \ for $h\leq j.$
\end{minipage}\end{equation}

Note that if in~\eqref{d.dv} we put $x=x_{h,j},$ $y=x_j,$ $z=x_{j+1},$
we get
\begin{equation}\begin{minipage}[c]{35pc}\label{d.dxhjxhj+1}
$d(x_{h,j},\,x_{h,j+1})\ \leq\ d(x_j,\,x_{j+1}).$
\end{minipage}\end{equation}

Also, for $h\leq j\leq k,$ the triangle inequality (applied
$k-j-1$ times) gives $d(x_{h,j},\,x_{h,k})\leq
\sum_{j\leq\ell<k}\,d(x_{h,\ell},\,x_{h,\ell+1}).$
Applying~\eqref{d.dxhjxhj+1} to each term of this summation, we get
\begin{equation}\begin{minipage}[c]{35pc}\label{d.dxhjxhk}
$d(x_{h,j},\,x_{h,k})\ \leq
\ \sum_{j\leq\ell<k}\,d(x_{\ell},\,x_{\ell+1}).$
\end{minipage}\end{equation}

In particular, for each $j,$ the distance from $x_{h,j}$ to any of the
later terms $x_{h,k}$ is
$\leq\sum_{j\leq\ell<\infty}\,d(x_\ell,\,x_{\ell+1}).$
As $j\to\infty,$ this sum approaches $0,$ so (still
for fixed $h)$ the elements $x_{h,j}$
$(j=h,\,h{+}1,\dots)$ form an increasing Cauchy sequence.
By~\eqref{d.incrCauchy} this sequence will converge; let
\begin{equation}\begin{minipage}[c]{35pc}\label{d.xhoo}
$x_{h,\infty}\ =\ \lim_{j\to\infty}\,x_{h,j}.$
\end{minipage}\end{equation}

Note that, by~\eqref{d.dxhjxhk},~\eqref{d.xhoo},
and the continuity of $d$ in the topology it defines, we have
\begin{equation}\begin{minipage}[c]{35pc}\label{d.xhjxhoo}
$d(x_{h,j},\,x_{h,\infty})\ \leq
\ \sum_{j\leq\ell<\infty}\,d(x_{\ell},\,x_{\ell+1})$ \ for $h\leq j.$
\end{minipage}\end{equation}

Note next that
\begin{equation}\begin{minipage}[c]{35pc}\label{d.xhjxij}
$x_{h,j}\ \geq\ x_{i,j}$ \ for $h\leq i\leq j.$
\end{minipage}\end{equation}
I claim that this implies that
\begin{equation}\begin{minipage}[c]{35pc}\label{d.xhooxioo}
$x_{h,\infty}\ \geq\ x_{i,\infty}$ \ for $h\leq i.$
\end{minipage}\end{equation}
Indeed,~\eqref{d.xhjxij} and~\eqref{d.xhooxioo} are respectively
equivalent to the conditions $d(x_{h,j},\,x_{h,j}{\vee}\,x_{i,j})=0$ and
$d(x_{h,\infty},\,x_{h,\infty}{\vee}\nolinebreak\, x_{i,\infty})=0,$
and the latter can be obtained from the former using~\eqref{d.xhoo}.

So the elements $x_{h,\infty}$ $(h=0,1,\dots)$
form a decreasing sequence.
I claim that this sequence, too, is Cauchy; in fact, that
\begin{equation}\begin{minipage}[c]{35pc}\label{d.dxhooxioo}
$d(x_{h,\infty},\,x_{i,\infty})\ \leq
\ \sum_{h\leq\ell<i}\,d(x_\ell,\,x_{\ell+1})$ \ for $h\leq i.$
\end{minipage}\end{equation}
Namely, by essentially the same
argument used to prove~\eqref{d.dxhjxhk},
one sees that for every $j\geq i,$
$d(x_{h,j},\,x_{i,j})\leq \sum_{h\leq\ell<i}\,d(x_\ell,\,x_{\ell+1});$
and by continuity of $d,$ this again carries over to the limit as
$j\to\infty.$
Hence by~\eqref{d.decrCauchy}, the terms
of~\eqref{d.xhooxioo} converge, and we can define
\begin{equation}\begin{minipage}[c]{35pc}\label{d.xoo,oo}
$x_{\infty,\infty}\ =\ \lim_{h\to\infty}\,x_{h,\infty}.$
\end{minipage}\end{equation}

Finally, note that for every $h\geq 0,$
\begin{equation}\begin{minipage}[c]{35pc}\label{d.dxhx00,00}
$d(x_h,\,x_{\infty,\infty})\ =\ d(x_{h,h},\,x_{\infty,\infty})\\[.3em]
\hspace*{3em}\leq\ d(x_{h,h},\,x_{h,\infty})+
d(x_{h,\infty},\,x_{\infty,\infty})\\[.3em]
\hspace*{3em}\leq\ \sum_{h\leq\ell<\infty}\,d(x_\ell,\,x_{\ell+1})\,+
\,\sum_{h\leq\ell<\infty}\,d(x_\ell,\,x_{\ell+1})\\[.3em]
\hspace*{3em}=\ 2\,\sum_{h\leq\ell<\infty}\,d(x_\ell,\,x_{\ell+1}),$
\end{minipage}\end{equation}
and that this sum approaches $0$ as $h\to\infty.$
Hence the $x_h$ converge,
\begin{equation}\begin{minipage}[c]{35pc}\label{d.xh>xoo}
$\lim_{h\to\infty}\,x_h\ =\ x_{\infty,\infty},$
\end{minipage}\end{equation}
completing the proof of the theorem.
\end{proof}

The first assertion of the following corollary clearly includes
the completeness result called on in~\S\S\ref{S.Bool}-\ref{S.nonBool}.
The remaining two assertions are further generalizations.

\begin{corollary}\label{C.meas}
Let $M$ be a measure space of finite total measure, and $B$
the Boolean ring of measurable subsets of $M$
modulo sets of measure zero.
For $S$ a measurable subset of $M,$ let $[S]$ denote its image in~$B.$
Then $B$ is complete with respect to the metric
\begin{equation}\begin{minipage}[c]{35pc}\label{d.d_on_M}
$d([S],\,[T])\ =\ \mu(S+T),$
\end{minipage}\end{equation}

More generally, if $M$ is a measure space not necessarily of
finite total measure, and $C$ a positive real constant, and we define
\begin{equation}\begin{minipage}[c]{35pc}\label{d.d_C}
$d_C([S],\,[T])\ =\ \r{min}(\mu(S+T),\,C),$
\end{minipage}\end{equation}
then the Boolean ring $B$ is complete with respect to $d_C.$

Alternatively, if, in the latter
situation, we define $B_\r{fin}$ to be the {\em nonunital}
Boolean ring of measurable sets of {\em finite} measure
modulo sets of measure zero, then
$B_\r{fin}$ is again complete with respect to the metric
$d$ of~\eqref{d.d_on_M}.

In all these cases, the operations of our Boolean ring are continuous
in the metric named.
\end{corollary}

\begin{proof}
In each of these cases, the lattice operations on
our structure are easily shown to
satisfy~\eqref{d.dv}-\eqref{d.decrCauchy} with respect
to the indicated metric.
(In the case of the metric $d_C$ of~\eqref{d.d_C},
note that in any Cauchy sequence, all but finitely many terms must
have the property that their distances from later terms
are all $\leq C,$ so that the definition~\eqref{d.d_C}, applied
to those distances, reduces to~\eqref{d.d_on_M}.)
Hence Theorem~\ref{T.lat} gives completeness.
\end{proof}

The standard result mentioned
earlier, that $L^1$ of the unit interval (indeed, of any
measure space) is complete in its natural metric,
follows similarly, on
regarding $L^1$ as a lattice under pointwise max and min.

We remark that in any partially ordered
set with a metric, the conjunction of conditions~\eqref{d.incrCauchy}
and~\eqref{d.decrCauchy} above is easily shown equivalent to the
single condition that every Cauchy sequence whose members form
a chain under the partial ordering converges.
But the pair of conditions as stated seems easier to work with.
In particular, it is easy to see that it holds for measurable
sets modulo sets of measure zero in a measure space.

Concerning conditions~\eqref{d.dv} and~\eqref{d.dw},
note that these are equivalent to Lipschitz continuity of $\vee,$
respectively $\wedge,$ with Lipschitz constant~$1.$
One could generalize the proof of Theorem~\ref{T.lat}
to allow any Lipschitz constant;
in fact, I suspect that a version of the theorem
could be proved -- at the cost of more complicated arguments --
with these conditions weakened to say
that $\vee,$ respectively, $\wedge,$ is uniformly continuous;
equivalently, that there exists a function $u$ from the positive reals
to the positive reals satisfying
\begin{equation}\begin{minipage}[c]{35pc}\label{d.dlimu}
$\lim_{t\to 0}\,u(t)\ =\ 0$
\end{minipage}\end{equation}
such that
\begin{equation}\begin{minipage}[c]{35pc}\label{d.dvu}
$d(x\vee y,\,x\vee z)\ \leq\ u(d(y,z))$ $(x,y,z\in L),$
\end{minipage}\end{equation}
respectively,
\begin{equation}\begin{minipage}[c]{35pc}\label{d.dwu}
$d(x\wedge y,\,x\wedge z)\ \leq\ u(d(y,z))$ $(x,y,z\in L).$
\end{minipage}\end{equation}

The idea would be to choose from a general Cauchy sequence a
subsequence for which the distance between
$\!i\!$-th and $\!i{+}1\!$-st terms decreases rapidly enough
not only to make these distances
have a convergent sum, but to have the corresponding
property after taking into account the effect of the
$u$ in~\eqref{d.dvu} or~\eqref{d.dwu} under the iterated application
of that inequality in the proof.
But I don't know whether there are situations where metrics
arise that would make it worth trying to prove such a result.

\section{Counterexample: a natural lattice under a strange metric}\label{S.ceg}

Another pair of conditions weaker than~\eqref{d.dv} and~\eqref{d.dw},
which I at one point thought might be able to replace those two
hypotheses in Theorem~\ref{T.lat}, are
\begin{equation}\begin{minipage}[c]{35pc}\label{d.dxxvy}
$d(x,\,x\vee y)\ \leq\ d(x,y)$ $(x,y\in L),$
\end{minipage}\end{equation}
and
\begin{equation}\begin{minipage}[c]{35pc}\label{d.dxxwy}
$d(x,\,x\wedge y)\ \leq\ d(x,y)$ $(x,y\in L).$
\end{minipage}\end{equation}

One can in fact prove from~\eqref{d.dxxvy} that
\begin{equation}\begin{minipage}[c]{35pc}\label{d.dx0xdots}
$d(x_0,\,x_0\vee\dots\vee x_i)\ \leq
\ \sum_{0\leq\ell<i}\,d(x_\ell,\,x_{\ell+1})$\quad
$(x_0,\dots,x_i\in L).$
\end{minipage}\end{equation}
Indeed,~\eqref{d.dxxvy} gives (if $i>0)$
$d(x_0,\,x_0\vee x_1\vee\dots\vee x_i)\leq
d(x_0,\,x_1\vee\dots\vee x_i);$
by the triangle inequality, the right-hand side
is $\leq d(x_0,\,x_1)+d(x_1,\,x_1\vee\dots\vee x_i).$
The second of these terms is (if $i>1)$ similarly
bounded by $d(x_1,\,x_2)+d(x_2,\,x_2\vee\dots\vee x_i),$
and this procedure, iterated, gives~\eqref{d.dx0xdots}.
But one cannot similarly get
\begin{equation}\begin{minipage}[c]{35pc}\label{d.0i0j}
$d(x_0\vee\dots\vee x_i,\ x_0\vee\dots\vee x_j)\ \leq
\ \sum_{i\leq\ell<j}\,d(x_\ell,\,x_{\ell+1})$\quad
$(0\leq i\leq j,\ x_0,\dots,x_j\in L)$
\end{minipage}\end{equation}
as would be needed to carry out the argument used in the proof of
Theorem~\ref{T.lat}.

I give below examples showing that
that theorem in fact does not hold with~\eqref{d.dxxvy}
and~\eqref{d.dxxwy} in place of~\eqref{d.dv} and~\eqref{d.dw}.
We will first get an example for upper semilattices and~\eqref{d.dxxvy},
then note how to modify it to make the semilattice into a lattice.
Applying duality, one gets examples for
the remaining two cases of the theorem.

To start the construction,
let $M$ be any metric space, with metric $d_M,$ and for any
finite nonempty subset $S$ of $M,$ define its diameter,
\begin{equation}\begin{minipage}[c]{35pc}\label{d.diam}
$\r{diam}(S)\ =\ \max_{x,y\in S}\,(d_M(x,y)).$
\end{minipage}\end{equation}

Now let $L$ be the upper semilattice of all finite nonempty subsets
of $M,$ under the operation of union; and for $S, T\in L$ define
\begin{equation}\begin{minipage}[c]{35pc}\label{d.d_L}
$d_L(S,\,T)\ =\ \left\{ \begin{array}{cl}
0 & \mbox{if}\ S=T,\\
\r{diam}(S\cup T) & \mbox{if}\ S\neq T.
\end{array}\right.$
\end{minipage}\end{equation}

It is straightforward that $d_L$ is a metric on $L;$
the only step requiring thought is the triangle inequality
$d_L(S,U)\leq d_L(S,T)+d_L(T,U)$ in the case where
the three sets $S,$ $T$ and $U$ are distinct and the maximum
defining the left-hand side of the
desired inequality is given by distance between
some $x\in S$ and some $y\in U.$
In that case, taking {\em any} $z\in T,$ one sees that
$d_L(S,U) = d_M(x,y) \leq d_M(x,z)+d_M(z,y)\leq d_L(S,T)+d_L(T,U),$
as required.

Under this metric,~\eqref{d.dxxvy} is also immediate:
writing that relation as $d(S,\,S\cup T)\leq d(S,T),$
we see that unless $T\subseteq S,$ the two sides both equal
$\r{diam}(S\cup T),$
while if $T\subseteq S,$ the left-hand side is zero.

I claim next that $L$ has no infinite strictly increasing
Cauchy sequences.
Indeed, given $S_0\subsetneqq S_1\subsetneqq \dots\in L,$ the
set $S_1$ must have more than one element, hence have
nonzero diameter; and from~\eqref{d.d_L} we see that
for every $i\geq 1,$ $d_L(S_i,S_{i+1})\geq \r{diam}(S_1),$ so the
distances between successive terms of the sequence do not
approach~$0.$
$L$ also has no infinite strictly decreasing Cauchy sequences,
since any strictly decreasing sequence
of sets starting with a finite set is finite.
So, trivially,~\eqref{d.incrCauchy} and~\eqref{d.decrCauchy} hold.

Note also that every non-singleton $S\in L$ has distance
at least $\r{diam}(S)$ from every other element of $L,$ so
it is an isolated point.
It follows that the set of non-singleton elements of $L$ is open
in $L,$ so the set of singleton elements is a closed set, which
is easily seen to be isometric to $M:$ $d_L(\{x\},\{y\})=d_M(x,y).$

Hence if we take for $M$ a non-complete metric space, then
a non-convergent Cauchy sequence in $M$ yields
a non-convergent Cauchy sequence in $L.$
So the semilattice $L$ is non-complete, despite
satisfying~\eqref{d.dxxvy},~\eqref{d.incrCauchy}
and~\eqref{d.decrCauchy}.

To get an example which is a lattice,
we pass from $L$ as above to~$L'=L\cup\{\emptyset\},$
which is clearly a lattice under union and intersection.
The only problem is how to extend the metric $d_L$ to $L'.$
We may in fact use {\em any} extension of this
metric to that set that does not sabotage
the non-completeness of~$L;$ for~\eqref{d.dxxvy} holds automatically
when $x$ and $y$ are comparable, and $\emptyset$ is
comparable to every element of $L'.$
So, for instance, we might choose a fixed $P\in L,$ and define
\begin{equation}\begin{minipage}[c]{35pc}\label{d.dempty}
$d_{L'}(\emptyset,S)\ =\ 1+d_L(P,S)$ \quad for all $S\neq\emptyset.$
\end{minipage}\end{equation}
Since this makes $\emptyset$ an isolated point, it leaves
the image of $M$ in $L'$ closed, so $L'$ remains non-complete.

We remark that the join operation of $L,$ and hence of $L',$
is in general discontinuous; for given any non-eventually-constant
sequence $x_0, x_1,\dots$ in $M$ that approaches a limit $y\in M,$
we know that $\{x_0\}, \{x_1\},\dots$ approach $\{y\}$ in $L;$
but for any $z\neq y,$ if we apply $-\vee\{z\}$ we get the sequence
$\{x_0,z\}, \{x_1,z\},\dots,$ which cannot approach the
isolated point $\{y,z\}.$

So is it plausible that continuity of the meet and
join operations, combined
with~\eqref{d.dxxvy},~\eqref{d.incrCauchy} and~\eqref{d.decrCauchy},
would imply completeness of a metric lattice?
Still no:  if we apply the above construction of $L$ with $M$
taken to be a discrete non-complete metric space (e.g.,
$\{n^{-1}\mid n\geq 1\}\subseteq[0,1]),$ then $L$ is also discrete.
(The only elements we don't already know are
isolated are the singletons $\{x\}$ for $x\in M;$
but taking $\varepsilon$ such that the ball of radius
$\varepsilon$ about $x$ in $M$ contains no other points, we find that
the ball of radius
$\varepsilon$ about $\{x\}$ in $L$ also contains no other points.)
Hence $L',$ metrized as in~\eqref{d.dempty}, is also discrete;
and any operation on a discrete space is
continuous, though $L'$ is, we have shown, non-complete.

\section{Open questions}\label{S.open}

I have not examined the question of whether the {\em conjunction}
of~\eqref{d.dxxvy} and~\eqref{d.dxxwy} might somehow force a metrized
lattice satisfying~\eqref{d.incrCauchy}
and~\eqref{d.decrCauchy} to be complete as a metric space.
(In the discrete $L'$ constructed above satisfying~\eqref{d.dxxvy},
the inequality~\eqref{d.dxxwy},
i.e., $d_{L'}(S,\,S\cap T)\leq d_{L'}(S,\,T),$ holds when
$S\cap T\neq\emptyset,$ but not, in general, when $S\cap T=\emptyset.)$

The referee has asked whether Gleason's original question
has the same answer if restricted to the case
where $R$ an integral domain.
I do not know whether this is so, with or without the
assumption that $R$ is complete in the given topology.
It seems an interesting question.

\section{Acknowledgements}\label{S.ackn}
I am indebted to David Handelman and Hannes Thiel
for pointing me to results in the literature on
completeness of metric spaces arising in measure theory,
and to the referee for several useful suggestions.

\end{document}